\documentclass[12pt]{amsart}
\usepackage{amsmath}
\usepackage{amssymb}
\usepackage{amsthm}
\usepackage{amsfonts}
\DeclareMathOperator{\dist}{\text{\rm dist}}

\newcommand{\ka}{{\kappa}}
\newcommand{\h}{{\mathcal H}}

\newtheorem{claim}{claim}[section]
\newtheorem{theorem}[claim]{Theorem}

\newtheorem{lemma}[claim]{Lemma}

\theoremstyle{remark}\newtheorem*{remark}{Remark}
\theoremstyle{definition}

\def\cw{{\mathcal W}}

\def\cs{{\mathcal S}}

\def\dist{{\rm dist\,}}

\newcommand{\Bh}{\mathbf{B}(\mathcal{H})}
\newcommand{\Rf}{\operatorname{Ref}}
\def\b{\mathbf{B}}
\def\Fk{\mathbf{F}_k}
\def\F1{\mathbf{F}_1}
\def\A11{{\mathbb{A}_1(1)}}
\def\B1h{\mathbf{B}_1 (\mathcal{H})}

\numberwithin{equation}{section}

\title[ Power partial isometries]{$2$--hyperreflexivity and hyporeflexivity\\ of power partial isometries}
\author{Kamila Piwowarczyk}
\author{Marek Ptak}
\address{Kamila Piwowarczyk, Department of Applied  Mathematics, University of Agriculture, ul. Balicka 253c, 30-198 Krak\'ow, Poland}
\email{kamila.piwowarczyk@ur.krakow.pl}
\address{Marek Ptak, Department of Applied Mathematics, University of Agriculture, ul. Balicka 253c, 30-198 Krak\'ow, Poland and
Institute of Mathematics, Pedagogical University, ul.
Podchor\c{a}\.{z}ych 2, 30-084 Krak\'ow, Poland}
\email{rmptak@cyf-kr.edu.pl}

\begin{document}

\maketitle

{\bf Abstract.} Power partial isometries are not always hyperreflexive neither reflexive. In the present paper it will be shown that  power partial isometries are always hyporeflexive and $2$--hyperreflexive.



\section{Introduction}
Reflexivity and  hyperreflexivity of operator algebras on Hilbert spa\-ces is connected with the problem of existence of a nontrivial invariant subspace.
An algebra of operators is reflexive \cite{[Sa]} if it has so many (common) invariant subspaces that they determine the algebra. It means that if any operator leaves invariant all subspaces which are invariant for all operators from the algebra, then it has to belong to the algebra. Equivalently, rank one operators contained in the preannihilator of an algebra  generate the whole preannihilator. Hyporeflexivity \cite{[HN]} of an algebra of operators (weaker property than reflexivity) means that if any operator from the commutant of the given algebra leaves invariant all subspaces, which are invariant for all operators from the algebra, then it  has to belong to the algebra.
An algebra of operators is hyperreflexive \cite{[Ar1]} (much stronger property than reflexivity) if the usual distance from any operator to the algebra
can be controlled by the distance given by rank one operators.
Replacing rank one operators by operators of rank at most $k$ in the corresponding conditions we obtain
the concepts of  $k$--reflexivity \cite{[Az]} and  $k$--hyperreflexivity \cite{[KP]} as natural generalizations of reflexivity and  hyperreflexivity.

A power partial isometry is an operator for which  all its powers are partial isometries.
In \cite{[ALMP]} full characterization of reflexivity of an algebra generated by  completely non--unitary power partial isometries was given.
In \cite{[PP]} it was shown that the same conditions given in \cite{[ALMP]} characterize hyperreflexive algebras generated by power partial isometries.
In the present paper we will show that algebras generated by power partial isometries are hyporeflexive, $2$--reflexive and even $2$-hyperreflexive.

\section{Preliminaries}

Let $\Bh$ denote the algebra of all bounded linear operators on a~complex separable Hilbert space $\h$.
For a cardinal number $d$ let $\h^{(d)}$ denote the orthogonal sum
of $\h$ with itself $d$ times. If $T\in \Bh$, then $T^{(d)}$ is the orthogonal sum of $T$ with itself $d$ times and for
$\cs\subset \Bh$ we denote $\cs^{(d)}=\{T^{(d)}\colon T\in \cs\}$.
Duality between trace class operators $\B1h$ and the
algebra $\Bh$ is given by trace, i.e.
$\langle T,f \rangle = \text{tr}\,(Tf)$ for $T \in \Bh,\; f \in \B1h$.
By $\Fk(\h)$ we denote the set of operators of rank at most $k$, $k \in \mathbb{N}$.
For a subset $\cs \subset \Bh$  by $\cs_{\bot}$ we
denote the preannihilator of $\cs$, i.e.
$\cs_{\bot} =\{f\in \B1h: \langle T, f\rangle =0 \text{ for all } T\in \cs\}$.

Let $\cs \subset \Bh$ be a subspace (i.e. a norm--closed linear manifold) and let $T\in \Bh$.
A subspace $\cs$ is \textit{reflexive} (\cite{[LS]}) if $\cs = \Rf \cs\overset{df}= \{ A\in \Bh: Ax \in [\cs x] \,\text{ for all}\, x\in\h \}$. (By $[\mathcal{M}]$ we denote the smallest closed subspace containing $\mathcal{M}$, in the appropriate space and topology.)
Longstaff in \cite{[Lo]} proved that a weak$^*$--closed subspace $\cs \subset \Bh$ is reflexive if and only if $\cs_\bot=[\cs_\bot \cap \F1(\h)]$.
A subspace $\cs$ is $k$--\textit{reflexive} \cite{[Az]} if $\cs^{(k)}$ is reflexive in $\b(\h^{(k)})$.
Kraus and Larson in \cite {[KL]} gave equivalent condition, namely a weak$^*$--closed subspace $\cs \subset \Bh$ is $k$--reflexive if and only if $\cs_\bot=[\cs_\bot \cap \Fk(\h)]$.

If $T\in\Bh$ by $\dist (T,\cs)$ we denote the usual distance from $T$ to $\cs$, namely $\dist (T,\cs)=\inf\{\|T-S\|: S\in \cs\}$. In what follows we will also consider the distances $\alpha_k(T,\cs)= \sup\{|\langle T, f \rangle | : f \in \cs_{\bot} \cap \Fk(\h), \|f\|_1 \leqslant 1 \}$, see \cite{[KP]}.
Recall that $\alpha_k(T,\cs) \leqslant \dist (T,\cs)$, $k\in\mathbb{N}$. A subspace $\cs$ is called $k$--\textit{hyperreflexive}
(\cite{[Ar1], [Ar2], [KP], [KL]})
if there is a constant $\ka > 0$ such that \[\dist (T,\cs) \leqslant \kappa\, \alpha_k (T,\cs)\quad\text{for all\ }T\in\Bh.\]
The infimum of all constants $\ka$ fulfilling this inequality
is called the $k$--\textit{hyperreflexivity constant} and denoted by $\ka_k(\cs)$.
We omit the letter $k$ if $k=1$ and say that $\cs$ is hyperreflexive.
An operator $A \in \Bh$ is called $k$--\textit{reflexive} ($k$--\textit{hyperreflexive}) if $\cw (A)$ is $k$--reflexive ($k$--hyperreflexive) where $\cw (A)$ denotes the smallest algebra containing polynomials in $A$ and closed in the weak operator topology.

Recall that the  \textit{unilateral shift} is the  operator  $a_s \in \mathbf{B}(l_+^2)$ defined as $a_s(x_0, x_1, \dots)=(0, x_0, x_1, \dots)$.
The \textit{backward shift} $a_c$ is its adjoint $a_c(x_0, x_1, \dots)=(x_1, x_2, \dots)$.
For a $k$-dimensional Hilbert space $H_k$ (isomorphically identified with $\mathbb{C}^k$) a {\it truncated shift} (Jordan block) $a_k \in \mathbf{B}(H_k)$ {\it of order} $k$, $1 \leqslant k < \infty$, is defined as $a_k(x_0, x_1, \dots, x_{k-1})=(0, x_0, x_1, \dots, x_{k-2})$.

Recall that $V\in \Bh$ is a {\it power partial isometry } if all its powers $V^n, n\in \mathbb{N}$, are partial isometries. The operators  $a_s, a_c, a_k$ are examples of power partial isometries. Moreover, they appear in the model of a power partial isometry. Recall after \cite{[HW]}

\begin{theorem} \label{hwdecompo}
Let $V\in \Bh$ be a power partial isometry. There exist subspaces
$\h_u(V)$, $\h_s(V)$, $\h_c(V)$, $\h_t(V) \subset \h$ such that
$\h_u(V)$, $\h_s(V)$, $\h_c(V)$, $\h_t(V)$ reduce $V$ and \\
$V_u= V|_{\h_{u}(V)}$ is a unitary operator,\\
$V_s= V|_{\h_{s}(V)}$ is a unilateral shift of arbitrary multiplicity,\\
$V_c= V|_{\h_{c}(V)}$ is a backward shift of arbitrary multiplicity,\\
$V_t= V|_{\h_{t}(V)}$ is possibly infinite orthogonal sum of truncated shifts
and
\begin{equation}\label{decompo}
V=V_u \oplus V_s \oplus V_c \oplus V_t.
\end{equation}
This decomposition is  unique.
\end{theorem}

In the following paper the theorem above will be the starting point in the main proofs. As we can realize "the proper" behaviour of reflexivity and hyperreflexivity as to orthogonal sums and heredity to subspaces  will be needed.

\begin{remark} \label{2}
In \cite{[KP],[KL],[LS]} may be found theorems, which deal with heredity of hyperreflexivity for subspaces and property $\mathbb{A}_1(r)$. This results were presented in our context in \cite[Proposition 2.2]{[PP]}.
\end{remark}

\begin{remark} \label{3}
Combining \cite[Theorem 6.16]{[Ha1]}, \cite[Theorem 5.1]{[KP]}, \cite[Theorems 3.8, 4.1]{[HN]} we get theorem, which deals with orthogonal sums of algebras and operators in the context of  hyperreflexivity and property $\mathbb{A}_1(r)$. The combined version can be found in \cite[Proposition 2.3]{[PP]}.
\end{remark}

In both Remarks above property $\mathbb{A}_1(r)$ were used.
Recall after \cite{[BFP]} that linear manifold $\cs \subset \Bh$ has \textit{property $\mathbb{A}_1$} if for any weak$^*$--continuous functional $\phi$ on $\cs$ there are $x$, $y\in \h$  such that $\phi(S)=\text{tr}(S\,(x\otimes y))$ for all $S \in \cs$. (By $(x \otimes y)$ we denote rank one operator  defined as $(x \otimes y)z=(z,y)x$ for $z \in \h$.) It is said that $\cs$ has \textit{property $\mathbb{A}_1(r)$}, $r\geqslant 1$, if $\cs$ has property $\mathbb{A}_1$ and for any $\varepsilon > 0$ vectors $x,y$ can be chosen such that $\| x\otimes y\|_1 \leqslant (r+\varepsilon) \|\phi\| $.

\section{Power partial isometries are $2$--hyperreflexive}

\begin{theorem}\label{hreppi}
If $V\in \Bh$ is a power partial isometry, then $V$ is $2$--hyperreflexive.
\end{theorem}

\begin{proof}
Recall that $V=V_u \oplus V_s \oplus V_c \oplus V_t$ (see (\ref{decompo})).
An algebra $\cw(V_u)$ is hyperreflexive (see \cite[Lemma 3.1]{[Ro]}) and has property $\A11$ (see \cite[Proposition 60.1]{[Co2]}). Thus $\cw(V_u^{(2)})$ is hyperreflexive and has property $\A11$ (\cite[Proposition 2.3]{[PP]}). Similarly, since $\cw(V_s)$ is hyperreflexive, $\kappa(\cw(a_s))<11,4$  (see \cite{[Da]}, \cite{[KP1]}) and has property $\A11$ (\cite[Proposition 60.5]{[Co2]}), thus $\cw(V_s^{(2)})$ is hyperreflexive and has property $\A11$ (\cite[Proposition 2.3]{[PP]}).
Recall also that both hyperreflexivity (with the same hyperreflexivity constant) and property $\mathbb{A}_1(1)$ are kept, when we take the adjoint, hence $\cw(V_c^{(2)})$ is hyperreflexive and has property $\A11$ (see also \cite[Proposition 3.1]{[PP]}).
By  \cite[Proposition III.1.21]{[Ber]} we get that $\cw(V_t^{(2)})$ has property $\A11$. If $V_t=\oplus_{i=1}^m a_{k_i}$, then $\cw(V_t^{(2)})$ is reflexive, since the largest block $a_{k_m}$ appears at least twice in the decomposition (\ref{decompo}) (see \cite[Theorem 2]{[DF]}). Thus $\cw(V_t^{(2)})$ is hyperreflexive, since, in such a case, underlying Hilbert space $\h(V)^{(2)}$ is finite dimensional,  see \cite{[KL],[MP1]}. If $V_t=\oplus_{i=1}^{\infty} a_{k_i}$ then $\cw(V_t^{(2)})$ is also hyperreflexive (see \cite[Theorem 3.3, Proposition 3.1]{[PP]}). Let us note that
\begin{align*}
\cw(V^{(2)}) &= \cw(V_u^{(2)} \oplus V_s^{(2)} \oplus V_c^{(2)} \oplus V_t^{(2)}) \\
&\subset \cw(V_u^{(2)})  \oplus \cw(V_s^{(2)}) \oplus \cw(V_c^{(2)}) \oplus \cw(V_t^{(2)})
\end{align*}
and $\cw(V_u^{(2)})  \oplus \cw(V_s^{(2)}) \oplus \cw(V_c^{(2)}) \oplus \cw(V_t^{(2)})$ is hyperreflexive and has property $\A11$ (\cite[Proposition 2.3]{[PP]}). Thus using \cite[Proposition 2.2]{[PP]} we get  hyperreflexivity of $\cw(V^{(2)})$. By \cite[Theorem 3.5]{[KP]} we obtain $2$--hyperreflexivity of $\cw(V)$.
\end{proof}

\begin{remark}
It is worth to note that even if in the proof above we have shown hyperreflexivity
we have not got an estimation of $\kappa_2 (\cw(V))$. The main reason is that it is not known whether $\kappa (\cw(a_k \oplus a_k))$ is bounded independently on $k$.
\end{remark}

\begin{remark}
$2$--reflexivity of a power partial isometry is a weaker version of Theorem \ref{hreppi}. On the other hand, it can be proved directly using the similar technique as in the proof of Theorem \ref{hreppi}.
\end{remark}

\section{Power partial isometries are hyporeflexive}

Hyporeflexivity was introduced in \cite{[HN]}. We say that a commutative algebra $\cw \subset \Bh$ is {\it hyporeflexive} if \[\cw =  \cw' \cap \operatorname{Alg\,Lat} \cw,\] where $\cw'$ denotes a commutant of $\cw$.
Let us note that if a commutative algebra is reflexive, then it is hyporeflexive. An operator $T \in \Bh$ is {\it hyporeflexive}  if $\cw(T)$ is hyporeflexive. Note that every reflexive operator $B\in \Bh$ is hyporeflexive since an algebra $\cw(B)$ is commutative. Let us recall that an operator acting on a finite dimensional Hilbert space is always hyporeflexive (see \cite[Theorem 10]{[BF]}).

Now we prove two technical lemmas.

\begin{lemma}\label{comm}
Let $\h=\oplus_{n \in \mathbb{N}} \h_{n}$ be an orthogonal sum of Hilbert spaces.
Let us consider an operator $\oplus_{n \in \mathbb{N}} A_n \in \oplus_{n \in \mathbb{N}} \b(\h_n)\subset \b(\h)$.
Then \[(\cw(\oplus_{n \in \mathbb{N}} A_n))' \cap (\oplus_{n \in \mathbb{N}} \b(\h_n)) =  (\oplus_{n \in \mathbb{N}} \cw(A_n))' \cap (\oplus_{n \in \mathbb{N}} \b(\h_n)).\]
\end{lemma}

\begin{proof}
Clearly $\cw (\oplus_{n \in \mathbb{N}} A_n) \subset \oplus_{n \in \mathbb{N}} \cw(A_n)$, thus by \cite[Proposition 12.2]{[Co2]} we get $$(\cw (\oplus_{n \in \mathbb{N}} A_n))' \supset (\oplus_{n \in \mathbb{N}} \cw(A_n))',\quad \text{ so}$$
$$(\cw(\oplus_{n \in \mathbb{N}} A_n))' \cap (\oplus_{n \in \mathbb{N}} \b(\h_n))
\supset (\oplus_{n \in \mathbb{N}} \cw(A_n))' \cap (\oplus_{n \in \mathbb{N}} \b(\h_n)).$$

Let us take
$T=\oplus_{n \in \mathbb{N}} T_n \in (\cw(\oplus_{n \in \mathbb{N}} A_n))'
\cap (\oplus_{n \in \mathbb{N}} \mathbf{B}(\h_n)).$
To prove the converse inclusion
 we should check that
\begin{equation}\label{3.4}
(\oplus_{n \in \mathbb{N}} T_n)(\oplus_{n \in \mathbb{N}} B_n) =
(\oplus_{n \in \mathbb{N}} B_n)(\oplus_{n \in \mathbb{N}} T_n)
\end{equation}
for $\oplus_{n \in \mathbb{N}} B_n \in \oplus_{n \in \mathbb{N}}\cw(A_n)$.

Let $p$ be a polynomial. Then
\begin{align*}
 (\oplus_{n \in \mathbb{N}} T_n)\  (p(A_1)\oplus (\oplus_{n \neq 1} 0))\ =\ &\
(\oplus_{n \in \mathbb{N}} T_n)\  p(\oplus_{n \in \mathbb{N}} A_n) \ (I\oplus (\oplus_{n \neq 1} 0))\\=\ &\ p(\oplus_{n \in \mathbb{N}} A_n)\  (\oplus_{n \in \mathbb{N}} T_n) \ (I\oplus (\oplus_{n \neq 1} 0))\\=\ &\ p(\oplus_{n \in \mathbb{N}} A_n)\  (I\oplus(\oplus_{n \neq 1} 0)) \ (\oplus_{n \in \mathbb{N}} T_n)\\=\ &\
(p(A_1)\oplus(\oplus_{n \neq 1} 0))\  (\oplus_{n \in \mathbb{N}} T_n).
\end{align*}
Let us take $B_1 \in \cw(A_1)$. There is a net of polynomials $p_\eta$ such that  $p_\eta(A_1)$ converges in the weak operator topology to $B_1$.  Then for $\oplus_ {n\in \mathbb{N}}h_n,\, \oplus_ {n\in \mathbb{N}}g_n\in \h$ we have
\begin{align*}
\langle (\oplus_{n \in \mathbb{N}} T_n)&\ (B_1\oplus (\oplus_{n \neq 1} 0))\  (\oplus_{n\in \mathbb{N}} h_n)\,,\, (\oplus_ {n\in \mathbb{N}}g_n) \rangle \\
=\ &\langle (B_1\oplus (\oplus_{n \neq 1} 0))\  (\oplus_{n\in \mathbb{N}} h_n)\, ,\,(\oplus_{n \in \mathbb{N}} T_n^*)\ (\oplus_ {n\in \mathbb{N}}g_n) \rangle\\=\ &\langle B_1 h_1, T_1^*g_1\rangle=\lim \,\langle p_\eta(A_1) h_1, T_1^*g_1\rangle\\=\ &\lim \,\langle (\oplus_{n \in \mathbb{N}} T_n)\ (p_\eta(A_1)\oplus (\oplus_{n \neq 1} 0))\ (\oplus_ {n\in \mathbb{N}}h_n)\, , \,(\oplus_ {n\in \mathbb{N}}g_n)\rangle \\=\ &\lim \,\langle(p_\eta(A_1)\oplus (\oplus_{n \neq 1} 0))\ (\oplus_{n \in \mathbb{N}} T_n)\ (\oplus_ {n\in \mathbb{N}}h_n)\, ,\,(\oplus_ {n\in \mathbb{N}}g_n)\rangle \\=\ &\lim \,\langle  p_\eta (A_1)\;T_1h_1,g_1\rangle =
\langle B_1 T_1h_1,g_1\rangle\\=\ &\langle (B_1\oplus(\oplus_{n \neq 1} 0))\ (\oplus_{n \in \mathbb{N}} T_n)\  (\oplus_{n\in \mathbb{N}} h_n)\, ,\,(\oplus_ {n\in \mathbb{N}}g_n) \rangle .
\end{align*}
Thus (\ref{3.4}) holds for $B_1 \oplus (\oplus_{n \neq 1} 0)$. Hence it is also fulfilled for $\oplus_{n=1}^k B_n \oplus (\oplus_{n >k} 0)$ for any  $k$. Since $\oplus_{n \in \mathbb{N}} B_n$ is the limit of $\oplus_{n=1}^k B_n \oplus (\oplus_{n >k} 0)$ ($k\to\infty$) in the weak and strong operator topology thus, as above, we get (\ref{3.4}) for  $\oplus_{n \in \mathbb{N}} B_n$.
\end{proof}

\begin{lemma} \label{AlgLatW(A+B)}
Let $\h=\oplus_{n \in \mathbb{N}} \h_{n}$ be an orthogonal sum of Hilbert spaces.
Let us consider an operator $\oplus_{n \in \mathbb{N}} A_n \in \oplus_{n \in \mathbb{N}} \b(\h_n)\subset \b(\h)$.
Then \[\operatorname{Alg\,Lat} \cw (\oplus_{n \in \mathbb{N}} A_n) \subset \oplus_{n \in \mathbb{N}} \b(\h_n).\]
\end{lemma}

\begin{proof}
Since
$\cw(\oplus_{n \in \mathbb{N}} A_n) \subset \oplus_{n \in \mathbb{N}} \cw (A_n)$,
thus by \cite[Proposition 5.6]{[AKG]} we get
\begin{align*}
\operatorname{Alg\,Lat} \cw (\oplus_{n \in \mathbb{N}} A_n) & \subset \operatorname{Alg\,Lat} (\oplus_{n \in \mathbb{N}} \cw(A_n) ) \\ &= \oplus_{n \in \mathbb{N}} \operatorname{Alg\,Lat} \cw (A_n) \subset \oplus_{n \in \mathbb{N}} \b(\h_n).
\end{align*}
\end{proof}

Now we prove two main theorems of this section.

  \begin{theorem} \label{sumhypo}
Let $\h=\oplus_{n \in \mathbb{N}} \h_{n}$ be an orthogonal sum of Hilbert spaces.
Let $A_n \in \b(\h_n)$ be a hyporeflexive operator
and let $\cw(A_n)$  have property $\A11$ for any $n \in \mathbb{N}$.
Then $\cw(\oplus_{n \in \mathbb{N}} A_n)$ is hyporeflexive.
  \end{theorem}

  \begin{proof}
Let us note that
\begin{multline*}
(\cw (\oplus_{n \in \mathbb{N}} A_n))' \cap (\oplus_{n \in \mathbb{N}} \b(\h_n)) \cap \operatorname{Alg\,Lat} \cw (\oplus_{n \in \mathbb{N}} A_n)  \\ \subset
(\cw (\oplus_{n \in \mathbb{N}} A_n))' \cap (\oplus_{n \in \mathbb{N}} \b(\h_n)) \cap \operatorname{Alg\,Lat} (\oplus_{n \in \mathbb{N}} \cw (A_n)).
\end{multline*}
By Lemma \ref{AlgLatW(A+B)} we get that
\[\operatorname{Alg\,Lat} \cw (\oplus_{n \in \mathbb{N}} A_n)\subset \oplus_{n \in \mathbb{N}} \b(\h_n) \] and from \cite[Proposition 5.6]{[AKG]} it follows that
\begin{align*}
\operatorname{Alg\,Lat} (\oplus_{n \in \mathbb{N}} \cw (A_n)) = \oplus_{n \in \mathbb{N}}\operatorname{Alg\,Lat} \cw(A_n)  \subset \oplus_{n \in \mathbb{N}} \b(\h_n),
\end{align*}
thus
\begin{align*}
(\cw (\oplus_{n \in \mathbb{N}} A_n))' & \cap \operatorname{Alg\,Lat} \cw (\oplus_{n \in \mathbb{N}} A_n)   \\ & \subset (\cw (\oplus_{n \in \mathbb{N}} A_n))' \cap \operatorname{Alg\,Lat} (\oplus_{n \in \mathbb{N}} \cw (A_n)).
\end{align*}
Hence, using Lemma \ref{comm} we get
\begin{align*}
(\cw (\oplus_{n \in \mathbb{N}}& A_n))' \cap \operatorname{Alg\,Lat} \cw (\oplus_{n \in \mathbb{N}} A_n)  \\ & \subset (\cw (\oplus_{n \in \mathbb{N}} A_n))' \cap \operatorname{Alg\,Lat} (\oplus_{n \in \mathbb{N}} \cw (A_n))  \\ & = (\cw (\oplus_{n \in \mathbb{N}} A_n))' \cap (\oplus_{n \in \mathbb{N}} \b(\h_n)) \cap \operatorname{Alg\,Lat} (\oplus_{n\in \mathbb{N}} \cw (A_n))\\ &= (\oplus_{n \in \mathbb{N}} \cw(A_n))' \cap (\oplus_{n \in \mathbb{N}} \b(\h_n)) \cap \operatorname{Alg\,Lat} (\oplus_{n \in \mathbb{N}} \cw (A_n)) \\ &= (\oplus_{n \in \mathbb{N}} \cw(A_n))' \cap \operatorname{Alg\,Lat} (\oplus_{n \in \mathbb{N}} \cw (A_n)) = \oplus_{n \in \mathbb{N}} \cw (A_n).
\end{align*}
The last equality follows from the fact that since all operators $A_n$ are hyporeflexive, for any $n \in \mathbb{N}$, thus the algebra $\oplus_{n \in \mathbb{N}} \cw (A_n)$ is also hyporeflexive (see \cite[Theorem 6.2(1)]{[HN]}). Now
each algebra $\cw(A_n)$ has property $\A11$ (for any $n \in \mathbb{N}$),
thus the sum $\oplus_{n \in \mathbb{N}} \cw (A_n)$ has property $\A11$
(\cite[Proposition 2.3]{[PP]}). Property $\A11$ is hereditary (\cite[Proposition 2.2]{[PP]}), thus \[(\cw (\oplus_{n \in \mathbb{N}} A_n))' \cap \operatorname{Alg\,Lat} \cw (\oplus_{n \in \mathbb{N}} A_n)\] has property $\A11$. Using \cite[Theorem 6.2(3)]{[HN]} we get hyporeflexivity of $\cw(\oplus_{n \in \mathbb{N}} A_n)$.
\end{proof}

\begin{theorem}
Let $V\in \Bh$ be a power partial isometry. Then $V$ is hyporeflexive.
\end{theorem}

\begin{proof}
Let $V=V_u \oplus V_s \oplus V_c \oplus V_t$ (see (\ref{decompo})).
Since $V_u$ is a unitary operator, hence $\cw(V_u)$ is reflexive (see \cite[Theorem 9.21]{[RR]}), thus hyporeflexive and has property $\A11$ by \cite[Proposition 60.1]{[Co2]}.

The unilateral shift $a_s \in \b(l_+^2)$ is reflexive (see \cite{[Sa]}), thus also
hyporeflexive. Recall that $\cw(a_s)$ has property $\A11$ (see \cite[Proposition 60.5]{[Co2]}), thus using Theorem \ref{sumhypo} we get hyporeflexivity of $\cw(V_s)$. It also  has property $\A11$ by \cite[Proposition 2.3]{[PP]}. The backward shift $a_c \in \b(l_+^2)$ is reflexive and has property $\A11$ (since both properties are preserved by taking the adjoint of operator), thus in the same way as above we get hyporeflexivity and property $\A11$ of $\cw(V_c)$.

A truncated shift $a_k \in \b(H_k)$ is hyporeflexive,  since every operator acting on finite dimensional space is hyporeflexive, see \cite[Theorem 10]{[BF]}).
Moreover, $a_k$ has property $\A11$ (\cite[Proposition III.1.21]{[Ber]}).
Hyporeflexivity and property $\A11$ of $\cw (V_t)$ we get from Theorem \ref{sumhypo}.

Since $V=V_u \oplus V_s \oplus V_c \oplus V_t$,  the hyporeflexivity of $\cw(V)$ follows once again from Theorem \ref{sumhypo}.
\end{proof}

{\bf Acknowledgements.}
Research was financed by the Ministry of Science and Higher Education of Republic of Poland.


\begin{thebibliography}{10}

\bibitem{[Ar1]}
W. T. Arveson, {\itshape Interpolation problems in nest algebras}, J. Funct. Anal. {\bfseries 20}, (1975), 208--233.

\bibitem{[Ar2]}
W. T. Arveson, {\itshape Ten lectures on operator algebras}, CBMS Regional Conference Series 55, Amer. Math. Soc., Providence (1984).

\bibitem{[Az]}
E. A. Azoff, {\itshape $k$--reflexivity in finite dimensional subspaces}, Duke Math. J., {\bfseries 40} (1973), 821--830.

\bibitem{[AKG]}
E. A. Azoff, C. K. Fong, F. Gilfeather, {\itshape A reduction theory for non--self--adjoint operator algebras}, Trans. Amer. Math. Soc., {\bfseries 224} (1976), 351--366.

\bibitem{[ALMP]}
E. A. Azoff, W. S. Li, M. Mbekhta, M. Ptak, {\itshape On consistent operators and reflexivity}, Integr. Equ. Oper. Theory, {\bfseries 71} (2011), 1--12.


\bibitem{[Ber]}
H. Bercovici, {\itshape Operator theory and arytmetic in $H^\infty$}, Mathematical Surveys and Monographs No. {\bfseries 26}, Amer. Math. Soc., 1988.

\bibitem{[BFP]}
H. Bercovici, C. Foias, C. Pearcy, {\itshape Dual Algebras with Applications to Invariant Subspaces and Dilation Theory}, CBMS Regional Conference Series 56, Amer. Math. Soc., Providence, 1985.

\bibitem{[BF]}
L. Brickman, P. A. Fillmore, {\itshape The invariant subspace lattice of a linear transformation}, Canad. J. Math. {\bfseries 19} (1967), 810--822.

\bibitem{[Co2]}
J. B. Conway, {\itshape A Course in Operator Theory}, Amer. Math. Soc., Providence, 2000.

\bibitem{[Da]}
K. R. Davidson, {\itshape The distance to the analytic Toeplitz operators}, Illinois J. Math. {\bfseries 31} (1987), 265--273.

\bibitem{[DF]}
J. A. Deddens, P. A. Fillmore, {\itshape Reflexive linear transformations}, Linear Algebra Appl. {\bfseries 10} (1975), 89--93.

\bibitem{[Ha1]}
D. Hadwin, {\itshape A general view of reflexivity}, Trans. Amer. Math. Soc. {\bfseries 344} (1994), 325--360.

\bibitem{[HN]}
D. Hadwin, E. A. Nordgren, {\itshape Subalgebras of reflexive algebras}, J. Operator Theory {\bfseries 7} (1982), 3--23.

\bibitem{[HW]}
P. R. Halmos, L. J. Wallen, {\itshape Powers of Partial Isometries}, J. Math. Mech. {\bfseries 19} (1970), 657--663.

\bibitem{[KP1]}
K. Kli\'s, M. Ptak, {\itshape Quasinormal operators are hyperreflexive}, Banach Center Publ. {\bfseries 67} (2005), 241--244.

\bibitem{[KP]}
K. Kli\'s, M. Ptak, {\itshape $k$--hyperreflexive subspaces}, Houston J. Math. {\bfseries 32} (2006), 299--313.

\bibitem{[KL]}
J. Kraus, D. R. Larson, {\itshape Reflexivity and distance formulae}, Proc. Lond. Math. Soc. {\bfseries 53} (1986), 340--356.

\bibitem{[LS]}
A. I. Loginov, V. I. Shulman, {\itshape Hereditary and intermediate reflexivity of W$^*$ algebras}, Math. USSR--Izv., {\bfseries 9} (1975), 1189--1201.

\bibitem{[Lo]}
W. E. Longstaff, {\itshape On the operation Alg Lat in finite dimensions}, Linear Algebra Appl. {\bfseries 27} (1979), 27--29.

\bibitem{[MP1]}
V. Müller, M. Ptak, {\itshape Hyperreflexivity of finite-dimensional subspace},
J. Funct. Anal., {\bfseries 218} (2005), 395--408.

\bibitem{[PP]}
K. Piwowarczyk, M. Ptak, {\itshape On the hyperreflexivity of power partial isometries}, Linear Algebra Appl., {\bfseries 437} (2012), 623-629.

\bibitem{[RR]}
H. Radjavi, P. Rosenthal, {\itshape Invariant Subspaces}, Springer-Verlag Berlin Heidberg New York, 1973.

\bibitem{[Ro]}
S. Rosenoer, {\itshape Distance estimates for von Neumann algebras}, Proc. Amer. Math. Soc. {\bfseries 86}(2) (1982), 248--252.

\bibitem{[Sa]}
D. Sarason, {\itshape Invariant subspaces and unstarred operator algebras}, Pacific J. Math. {\bfseries 17} (1966), 511--517.


\end{thebibliography}
\end{document}